\documentclass[12pt]{amsart}
\usepackage{tikz}
\usetikzlibrary{arrows.meta}
\usepackage{geometry}
\usepackage[english]{babel}
\usepackage{graphicx}
\usepackage{babelbib}
\usepackage{url}
\usepackage{hyperref}

\newtheorem{theorem}{Theorem}[section]
\newtheorem{lemma}[theorem]{Lemma}

\theoremstyle{definition}
\newtheorem{definition}[theorem]{Definition}

\theoremstyle{remark}
\newtheorem{remark}[theorem]{Remark}

\numberwithin{equation}{section}

\newcommand\spa{\operatorname{\text{span}}}
\newcommand\AUT{\operatorname{\sf AUT}}

\makeatletter
\newsavebox\myboxA
\newsavebox\myboxB
\newlength\mylenA
\newcommand*\xbar[2][0.75]{%
    \sbox{\myboxA}{$\m@th#2$}%
    \setbox\myboxB\null
    \ht\myboxB=\ht\myboxA%
    \dp\myboxB=\dp\myboxA%
    \wd\myboxB=#1\wd\myboxA
    \sbox\myboxB{$\m@th\overline{\copy\myboxB}$}
    \setlength\mylenA{\the\wd\myboxA}
    \addtolength\mylenA{-\the\wd\myboxB}%
    \ifdim\wd\myboxB<\wd\myboxA%
       \rlap{\hskip 0.5\mylenA\usebox\myboxB}{\usebox\myboxA}%
    \else
        \hskip -0.5\mylenA\rlap{\usebox\myboxA}{\hskip 0.5\mylenA\usebox\myboxB}%
    \fi}
\makeatother

\newcommand{\abs}[1]{\lvert#1\rvert}

\begin{document}

\title{A general bridge theorem for self-avoiding walks}

\author{Christian Lindorfer}
\address{\parbox{.8\linewidth}{Institut f\"ur Diskrete Mathematik,\\ 
Technische Universit\"at Graz,\\
Steyrergasse 30, A-8010 Graz, Austria\\ $\,$}}
\email{lindorfer@math.tugraz.at}

		  
\date{\today} 



\begin{abstract}
Let $X$ be an infinite, locally finite, connected, quasi-transitive graph without loops or multiple edges. A graph height function on $X$ is a map adapted to the graph structure, assigning to every vertex an integer, called height. Bridges are self-avoiding walks such that heights of interior vertices are bounded by the heights of the start- and end-vertex. The number of self-avoiding walks and the number of bridges of length $n$ starting at a vertex $o$ of $X$ grow exponentially in $n$ and the bases of these growth rates are called connective constant and bridge constant, respectively. We show that for any graph height function $h$ the connective constant of the graph is equal to the maximum of the two bridge constants given by increasing and decreasing bridges with respect to $h$. As a concrete example, we apply this result to calculate the connective constant of the Grandparent graph.
\end{abstract}

\maketitle
\vspace*{-0.9cm}
\section{Introduction and results}

Let $(X,\sim)$ be a simple, locally finite, connected, infinite graph consisting of a countable set of vertices $X$ and a symmetric neighbourhood relation $\sim$. We consider a walk on the graph as a sequence of vertices, where consecutive vertices in the walk must be adjacent in the graph and call it \emph{self-avoiding} (or a SAW), if no vertex of the graph is contained twice in this sequence. The length of a walk will be the length of the sequence reduced by one, so the number of ``steps" in the walk.\par

We shall consider only graphs which are quasi-transitive, i.e. the automorphism group of $X$ acts with finitely many orbits on $X$. Fix some vertex $o$ of $X$ and denote by $c_{n,o}$ the number of SAWs of length $n$ starting at $o$. The existence of
\begin{equation*}
\mu(X)=\lim_{n \rightarrow \infty} c_{n,o}^{1/n}
\end{equation*}
and its independence of the choice of $o$ were shown in \cite{MR0091568} by {\sc Hammersley} and the number $\mu(X)$ is called the \emph{connective constant} of the graph $X$.\par 
The study of self-avoiding walks, primarily in lattice graphs, was initiated by the chemist and Nobel laureate {\sc P. J. Flory} (see his book \cite{CR1000000}) as a tool for studying long-chain polymer growth. Physicists are especially interested in SAWs in the integer lattices $\mathbb{Z}^d$ for $d \geq 2$ and a lot of work has been dedicated to the subject. Good references are the monograph by {\sc Madras and Slade} \cite{MR2986656} and the lecture notes by {\sc Bauerschmidt et al.} \cite{MR3025395}. Although some good bounds are known for $\mu(\mathbb{Z}^d)$, the precise values for $d \geq 2$ are still unknown. \par

A highlight has been the very well received paper \cite{MR2912714} by {\sc Duminil-Copin and Smirnov}, who proved that the connective constant of the honeycomb lattice equals $\sqrt{2 + \sqrt{2}}$. In the final part of the proof they used the Hammersley-Welsh method, a decomposition of walks into bridges, which was introduced in \cite{MR0139535} and serves as the basis of our bridge theorem. \par 

A \emph{graph height function} $h$ on a quasi-transitive graph $X$ is a map assigning to every vertex $v$ of $X$ an integer $h(v)$, called the height of $v$. Graph height functions have to be adapted to the graph structure according to Definition \ref{def:ghf}. They are used to define bridges, which are self-avoiding walks $\pi=(v_0, \dots, v_n)$ satisfying
\begin{equation*}
h(v_0)<h(v_i) \leq h(v_n), \quad 0<i<n.
\end{equation*}
Similarly to the connective constant of a graph, one can now define the \emph{bridge constant} $\beta(X,h)$, the base of the exponential growth of the number of bridges $b_{n,o}$ of length $n$ starting at $o$, which is again independent of the choice of $o$.
\begin{equation*}
\beta(X,h)=\lim_{n \rightarrow \infty} b_{n,o}^{1/n}.
\end{equation*}\par

The definitions of graph height functions, bridges and the bridge constant were first given by {\sc Hammersley and Welsh} in \cite{MR0139535} for the graph $\mathbb{Z}^d$, where $h$ mapped any vertex to its first coordinate. In this paper the authors also invented the above mentioned method of decomposing self-avoiding walks into bridges and used it to show that the connective constant is equal to the bridge constant.
{\sc Grimmett and Li} then generalized graph height functions, bridges and the bridge constant to quasi-transitive graphs in \cite{MR3862645} and proved a version of the bridge theorem, stating that the bridge constant with respect to a \emph{unimodular} graph height function is equal to the connective constant of such a graph.\par
The goal of this paper is to generalize the bridge theorem, in the sense that we do not need graph height functions to be unimodular. Our main result Theorem \ref{thm:extbridgethm} states that for any graph height function $h$ on $X$ the connective constant $\mu(X)$ is equal to the maximum of the bridge constant $\beta(X,h)$ corresponding to $h$ and the bridge constant $\beta(X,-h)$ corresponding to the "reflected" graph height function $-h$.\par 

As a consequence of this theorem, all results discussed by {\sc Grimmett and Li} in section 5 of \cite{MR3862645} concerning locality of connected constants also hold in the case where the graph height functions are not unimodular. In particular they obtained conditions, under which the connective constants of sequences of graphs possessing unimodular graph height functions converges to the connective constant of a limit graph of this sequence. The proofs in the non-unimodular case work exactly the same after replacing corresponding results by the results obtained in this paper, so they will not be discussed here.\par

In Section \ref{sec:grandparent}, we provide a concrete example. The Grandparent graph was given by {\sc Trofimov} in \cite{MR811571} as an example of a connected, locally finite, transitive graph with a non-unimodular automorphism group, so it admits only non-unimodular graph height functions. We calculate the bridge constants with respect to the generic graph height function and use our version of the bridge theorem to obtain the connective constant of the Grandparent graph. This connective constant can also be obtained by using a different method described in \cite{AR005}.

\section{Terminology and preliminaries}
We consider a graph $(X,\sim)$ as a countable set of vertices $X$ together with a symmetric neighbourhood relation $\sim$ on $X$. We will usually write $X$ for the graph and omit $\sim$. Two vertices $u$ and $v$ are called \emph{adjacent} if $u \sim v$. The degree $deg_X(v)$ of a vertex $v$ in $X$ is the number of vertices of $X$ which are adjacent to $v$. We call $X$ \emph{locally finite}, if $deg_X(v)<\infty$ for every $v \in X$.\par

A \emph{walk} $\pi$ on a graph $X$ is a sequence of vertices $(v_0, v_1, \dots, v_n)$ of $X$ such that any two consecutive vertices of the sequence are adjacent in $X$. We denote by $\pi^-=v_0$ its first vertex, by $\pi^+=v_n$ its last vertex and by $\abs{\pi}=n$ the length of the walk. A walk $\pi$ is called \emph{self-avoiding} (or a SAW), if no vertex of $X$ occurs more than once in $\pi$. A graph $X$ is called connected if for any pair $(u,v)$ of vertices there is a walk on $X$ starting at $u$ and ending at $v$.\par
The \emph{automorphism group} $\AUT(X)$ of a  graph $X$ is the group of all permutations $\sigma$ on $X$ such that for all $u,v \in X$ it holds that $u \sim v$ if and only if $\sigma (u) \sim \sigma (v)$.
For a subgroup $\Gamma \leq \AUT(X)$ and some $v \in X$ we write $\Gamma v=\{\gamma v \mid \gamma \in \Gamma\}\subset X$ for the \emph{orbit} of $v$ under the action of $\Gamma$ and $\Gamma_v=\{\gamma \in \Gamma \mid \gamma v=v\} \subset \Gamma$ for the vertex stabilizer of $v$ in $\Gamma$. The group $\Gamma$ is said to act \emph{quasi-transitively} on $X$ if the action of $\Gamma$ on $X$ admits finitely many orbits and it acts \emph{transitively} if there is exactly one orbit. 
The graph $X$ is called \emph{transitive} (respectively \emph{quasi-transitive}) if its automorphism group $\AUT(X)$ acts transitively (respectively quasi-transitively) on $X$.\par
As we are interested in the growth rate of the number of SAWs of length $n$ for $n$ going to infinity, we only consider infinite graphs. We denote by $\mathcal{X}$ the set of all infinite, connected, locally finite and quasi-transitive graphs.\par
The following definition of graph height functions on graphs in $\mathcal{X}$ is taken from \cite{MR3862645}.
\begin{definition}\label{def:ghf}
Let $X \in \mathcal{X}$. A graph height function on $X$ is a pair $(h,\Gamma)$, where
\begin{enumerate}
\item[(i)] $h:X \rightarrow \mathbb{Z}$,
\item[(ii)] $\Gamma \leq \AUT(X)$ is a subgroup of graph automorphisms acting quasi-transitively on $X$ and $h$ is $\Gamma$-difference-invariant in the sense that
		\begin{equation*}
		h(\gamma v)-h(\gamma u)=h(v)-h(u) \quad \text{for all } \gamma \in \Gamma, \; u,v \in X,
		\end{equation*}
	\item[(iii)] for every $v \in X$, there exist $u,w \in X$ adjacent to $v$ such that 
	\begin{equation*}
	h(u) < h(v) < h(w).
	\end{equation*}
\end{enumerate} 
A graph height function $(h,\Gamma)$ is called \emph{unimodular} if the action of $\Gamma$ on $X$ is unimodular, i.e., if $\abs{\Gamma_u v} = \abs{\Gamma_v u}$ for all $u,v \in X$ with $v \in \Gamma u$.\\
Denote by $d=d(h,\Gamma)$ the smallest integer satisfying $h(u)-h(v)\leq d$ for all $u \sim v$. When talking about a graph height function $(h,\Gamma)$ we will often simply write $h$ and omit $\Gamma$.
\end{definition}

\begin{remark}
There are graphs in $\mathcal{X}$ which do not support any graph height functions. This is still true when considering the set of all Cayley graphs of finitely generated groups: it was shown in \cite{MR3743101}, that neither the Cayley graph of the Grigorchuk group nor the Cayley graph of the Higman group admits a graph height function.
\end{remark}

Let $X \in \mathcal{X}$ and $(h,\Gamma)$ be a graph height function on $X$. For any $v \in X$ we use the following notation for special sets of SAWs $(v=v_0,v_1, \dots,v_n)$ of length $n \geq 0$ on $X$ starting at $v$:
\begin{itemize}
\item $C_{n,v} \dots$ all SAWs.
\item $B_{n,v} \dots$ \emph{bridges}: $h(v)<h(v_i)\leq h(v_n)$ for all $i \in [1,n]$.
\item $\xbar{B}_{n,v} \dots$ \emph{reversed bridges}: $h(v)>h(v_i)\geq h(v_n)$ for all $i \in [1,n]$.
\item $H_{n,v} \dots$ \emph{half-space-walks} (HSW): $h(v)<h(v_i)$ for all $i \in [1,n]$.
\item $\xbar{H}_{n,v} \dots$ \emph{reversed half-space-walks}: $h(v)>h(v_i)$ for all $i \in [1,n]$.
\end{itemize}\par

By convention all of the above sets contain the walk of length $0$ consisting of the  single vertex $v$.
We denote by $c_{n,v},b_{n,v},\xbar{b}_{n,v},h_{n,v}$ and $\xbar{h}_{n,v}$ the cardinalities of the respective sets $C_{n,v},B_{n,v},\xbar{B}_{n,v},H_{n,v}$ and $\xbar{H}_{n,v}$. Note that because of symmetry, reversed bridges and reversed half-space-walks with respect to $h$ are bridges and HSWs with respect to the "reflected" height function $-h$, so we will state most results only for bridges and HSWs.\par

Let $\pi$ be a walk on $X$. The span of $\pi$ is defined as the maximal height difference of two vertices in $\pi$:
\begin{equation*}
\spa(\pi)=h_{max}(\pi)-h_{min}(\pi), \text{ where } 
\end{equation*}
\begin{equation*}
h_{max}(\pi)=\max_{v \in \pi} h(v), \quad h_{min}(\pi)=\min_{v \in \pi} h(v).
\end{equation*}
It is clear that for any bridge $\pi$, $ \, \spa(\pi)=h(v_n)-h(v_0)$.\par
 
The group $\Gamma$ acts quasi-transitively on $X$ and is $h$-difference invariant, so it is possible to define 
\begin{equation*}
c_n=\max_{v \in X}c_{n,v},\quad b_n=\min_{v \in X} b_{n,v}\quad \text{and} \quad \xbar{b}_n=\min_{v \in X} \xbar{b}_{n,v}.
\end{equation*}\par

Any SAW $(v=v_0, \dots, v_{n+m}) \in C_{n+m,v}$ can be decomposed into a pair $(v_0, \dots, v_n) \in C_{n,v}$ and $(v_n, \dots, v_{n+m}) \in C_{m,v_n}$ of SAWs. Picking $v$ such that $c_{n+m,v}=c_{n+m}$ results in
\begin{equation*}
c_{n+m}=c_{n+m,v}\leq c_{n,v} c_m \leq c_n c_m,
\end{equation*}
so $(c_n)_{n \geq 0}$ is a sub-multiplicative sequence. On the other hand the concatenation of the bridges $ (v=v_0, \dots, v_n) \in B_{n,v}$ and $(v_n, \dots, v_{n+m}) \in B_{m,v_n}$ results in the bridge $(v_0, \dots, v_{n+m}) \in B_{n+m,v}$. Picking $v$ such that $b_{n+m,v}=b_{n+m}$ yields
\begin{equation*}
b_n b_m \leq b_{n,v} b_m \leq b_{n+m,v} =b_{n+m}.
\end{equation*}
Fekete's Lemma on sub-additive sequences provides the existence of the limits
\begin{equation*}
\mu(X):=\lim_{n \rightarrow \infty} c_n^{1/n}, \quad \beta(X,h):=\lim_{n \rightarrow \infty} b_n^{1/n}, \quad \xbar{\beta}(X,h):=\lim_{n \rightarrow \infty} \xbar{b}_n^{1/n}.
\end{equation*}
Here $\mu(X)$ depends only on the underlying graph $X$ and is called the connective constant of $X$ and $\beta(X,h)$ and $\xbar{\beta}(X,h)$ depend on the graph and the chosen height function $h$ and are called the bridge constant and reversed bridge constant of $X$ with respect to $h$, respectively. We will usually omit the graph and the height function if they are clear and just write $\mu$, $\beta$ and $\xbar{\beta}$.\par
Trivially it holds that $b_{n}\leq c_{n}$, so we obtain
\begin{equation}\label{eq:betamuest}
b_n \leq \beta^n \leq \mu^n \leq c_n, \quad n \geq 0,
\end{equation}
and the similar statement for $\xbar{b}_n$ and $\xbar{\beta}$.
{\sc Hammersley} showed in \cite{MR0091568}, that
\begin{equation*}
\lim_{n \rightarrow \infty} c_{n,v}^{1/n} = \mu \quad \text{for every } v \in X
\end{equation*}
and {\sc Grimmett and Li} proved in \cite{MR3862645}  the similar statement
\begin{equation*}
\lim_{n \rightarrow \infty} b_{n,v}^{1/n} = \beta \quad \text{for every } v \in X.
\end{equation*}\par
We conclude that it is possible to obtain connected constant and bridge constant as the radius of convergence of the generating functions of self-avoiding walks and bridges starting at any vertex $v$, respectively, independent of the choice of $v$.
\section{The bridge theorem}

One of the main results of \cite{MR3862645} is the bridge theorem.

\begin{theorem}[Thm. 4.3 in \cite{MR3862645}]
Let $X \in \mathcal{X}$ possess an unimodular graph height function $(h,\Gamma)$. Then $\mu(X)=\beta(X,h)$.
\end{theorem}

Note that as a consequence, the bridge constant $\beta(X,h)$ does not depend on the choice of the unimodular graph height function $h$. However, there are simple examples showing that unimodularity is required in this theorem, one of them being the grandparent graph, which will be discussed in Section \ref{sec:grandparent}.\par 
The main result in this paper is the following extension of this bridge theorem, holding without the requirement of unimodularity. 

\begin{theorem}[General Bridge Theorem]\label{thm:extbridgethm}
Let $X \in \mathcal{X}$ be a graph possessing a graph height function $(h,\Gamma)$. Then 
\begin{equation}\label{eq:mubeta}
\mu(X)=\beta_{max}(X,h):=\max\{\beta(X,h),\xbar{\beta}(X,h)\}.
\end{equation}
\end{theorem}
One inequality is clear from (\ref{eq:betamuest}), we only need to show $\mu \leq \beta_{max}$.
For convenience we first provide a detailed proof of the transitive case which we then generalize to the quasi-transitive case.

Let $X \in \mathcal{X}$ be a graph and $(h,\Gamma)$ be a graph height function on $X$ and assume that the group $\Gamma$ acts transitively on $X$. Then the value of $c_{n,v}$ does not depend on $v$ and is therefore equal to $c_n$. Moreover elements of $\Gamma$ map bridges onto bridges and HSWs onto HSWs, implying that also $b_{n,v}$ and $h_{n,v}$ do not depend on $v$, so we can omit $v$ in the notation.\par

For simplicity we fix some vertex $o$ of $X$ with $h(o)=0$ and write $C_n, B_n, H_n$ and $\xbar{B}_n, \xbar{H}_n$ for the sets of SAWs, bridges, HSWs and their reversed versions starting at $o$, respectively. Moreover for every $v \in X$ we fix some element $\gamma_v \in \Gamma$ with $\gamma_v(o)=v$. The concatenation of two walks $\pi_1=(o, v_1, \dots, v_m)$, $\pi_2=(o,w_1,\dots,w_n)$ is defined as the walk  
\begin{equation*} 
\pi_1\pi_2:= (o,v_1, \dots, v_m, \gamma_{v_m} (w_1), \dots, \gamma_{v_m}(w_n)). 
\end{equation*}
Similarly, the decomposition of $\pi_1$ at $v_l$ is defined to provide the two walks $(o,v_1, \dots, v_l)$ and $(\gamma_{v_l}^{-1}(v_l), \dots, \gamma_{v_l}^{-1}(v_n))$, both of them starting at $o$.

Denote by $B_{n}(a)$ the set of bridges in $B_n$ having span $a \geq 0$ and by $b_n(a)$ the cardinality of this set. Note that $b_0(0)=1$ and for $d=d(h,\Gamma)$ from Definition \ref{def:ghf} it trivially holds that $b_n(a)=0$ for $a>d n$ because the height distance per step is at most $d$. It follows that
\begin{equation}\label{eq:bridgessum}
b_n=\sum_{a=0}^{dn} b_n(a).
\end{equation}

Take any HSW $\pi=(v_0,v_1,\dots,v_n)$ of length $n \geq 1$. We use the following iterative process to decompose $\pi$ into an alternating sequence of bridges and reversed bridges: Let $i_0=0$ and in step $j \geq 1$ define  
\begin{equation*}
a_j= \max_{i\in [i_{j-1},n]} \abs{h(v_i)-h(v_{i_{j-1}})},
\end{equation*}
and $i_j$ as the largest index in $[i_{j-1},n]$, where the maximum is attained. Then the sub-walk $\pi_j=(v_{i_{j-1}}, \dots, v_{i_j})$ of $\pi$ is a bridge if $j$ is odd and a reversed bridge if $j$ is even. Moreover by definition the span of $\pi_j$ is $a_j$.\par
By construction the span decreases in every step, so $a_1 > \dots > a_k>0$. We denote by $H_n(a_1, \dots, a_k)$ the set of HSWs in $H_n$ decomposing into an alternating sequence of bridges and reversed bridges of spans $a_1, \dots, a_k$ and by $h_n(a_1,\dots,a_k)$ its cardinality. It is clear that for $n \geq 1$
\begin{equation}\label{eq:hswinspans}
h_n=\sum_{k>0} \sum_{a_1> \dots > a_k>0} h_n(a_1, \dots, a_k).
\end{equation}
Moreover for $k=1$ it follows directly from the definition that $h_n(a_1)=b_n(a_1)$.

\begin{lemma}\label{lem:seqofhsws}
Let $n, k \geq 1$ and $a_1 > a_2 > \dots > a_k >0$. Then
\begin{equation}\label{eq:conseqofhsws} 
h_n(a_1, \dots, a_k) \leq \sum_{m=0}^{n} b_m(a_1+a_3+\dots) \xbar{b}_{n-m} (a_2+a_4+\dots).
\end{equation}
\end{lemma}
\begin{proof}
Let $\pi \in H_n(a_1, \dots, a_k)$. We use the decomposition described before to construct a pair $(\pi_+, \pi_-)$ consisting of a bridge and a reversed bridge, both starting at $o$. We begin by decomposing $\pi$ into the walks $\pi_1, \dots, \pi_k$ such that the span of $\pi_i$ is $a_i$ for every $i$ and $\pi_i$ is a bridge if $i$ is odd and a reversed bridge otherwise. Let $\pi_+=\pi_1\pi_3\dots$ be the concatenation of all $\pi_i$ which are bridges ($i$ odd) and $\pi_-=\pi_2\pi_4\dots$ be the concatenation the $\pi_i$ which are reversed bridges ($i$ even). This construction can be seen in Figure \ref{fig:concatbridges}.\par
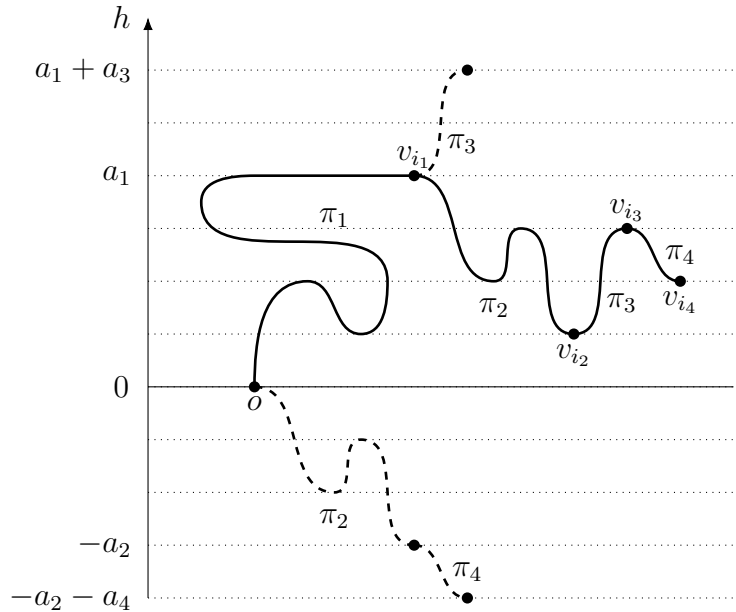
\begin{figure}[ht]
\centering
\begin{tikzpicture}
\begin{scope}[scale=0.7]
\foreach \i in {-4,-3,-2,-1,0,1,2,3,4,5,6}{
\draw[dotted] (0,\i*1) -- (11,\i*1);
}
\draw[-Latex] (0,-4) -- (0,7);
\draw (0,0) -- (11,0);
\fill (2,0) circle (3pt);
\fill (5,4) circle (3pt);
\fill (8,1) circle (3pt);
\fill (9,3) circle (3pt);
\fill (10,2) circle (3pt);
\fill (5,-3) circle (3pt);
\fill (6,6) circle (3pt);
\fill (6,-4) circle (3pt);

\node at (-0.5,7) {$h$};
\node at (-0.5,0) {0};
\node at (2,-0.3) {$o$};
\node at (5,4.35) {$v_{i_1}$};
\node at (8,0.6) {$v_{i_2}$};
\node at (9,3.35) {$v_{i_3}$};
\node at (10,1.6) {$v_{i_4}$};
\node[anchor=east] at (-0.1,4) {$a_1$};
\node[anchor=east] at (-0.1,6) {$a_1+a_3$};
\node[anchor=east] at (-0.1,-3) {$-a_2$};
\node[anchor=east] at (-0.1,-4) {$-a_2-a_4$};
\node at (3.5,3.2) {$\pi_1$};
\node at (6.5,1.5) {$\pi_2$};
\node at (8.9,1.6) {$\pi_3$};
\node at (10,2.5) {$\pi_4$};
\node at (3.5,-2.5) {$\pi_2$};
\node at (5.9,4.6) {$\pi_3$};
\node at (6,-3.5) {$\pi_4$};

\draw[line width=1pt] (2,0) to[out=90, in=180] (3,2) to[out=0, in=180] (4,1) to[out=0, in=-90] (4.5,2) to[out=90, in=-90] (1,3.5) to[out=90, in=180] (2,4) to[out=0, in=180] (5,4);
\draw[line width=1pt] (5,4) to[out=0, in=180] (6.5,2) to[out=0, in=180] (7,3) to[out=0, in=180] (8,1); 
\draw[line width=1pt] (8,1) to[out=0, in=180] (9,3); 
\draw[line width=1pt] (9,3) to[out=0, in=180] (10,2);

\draw[line width=1pt, dashed] (5,4) to[out=0, in=180] (6,6);
\draw[line width=1pt, dashed] (2,0) to[out=0, in=180] (3.5,-2) to[out=0, in=180] (4,-1) to[out=0, in=180] (5,-3);  
\draw[line width=1pt, dashed] (5,-3) to[out=0, in=180] (6,-4);
\end{scope}
\end{tikzpicture}
	\caption{Decomposition of a HSW $\pi$ into bridges and reversed bridges and construction of $\pi_+$ and $\pi_-$ (dashed).}
	\label{fig:concatbridges}
\end{figure}
Clearly $\pi_+$ is a bridge and its span is $a_1+a_3+\dots$ and $\pi_-$ is a reversed bridge and its span is $a_2+a_4+\dots$. Moreover from the knowledge of the sequence $a_1, \dots, a_k$ and the two walks $\pi_+$ and $\pi_-$ the original walk $\pi$ can be uniquely constructed, so the construction of the pair $(\pi_+, \pi_-)$ is injective. 
Let $m =\abs{\pi_1}+ \abs{\pi_3}+ \dots$ be the sum of lengths of the odd-index sub-walks $\pi_i$. Then $\pi_+ \in B_m(a_1+a_3+\dots)$ and $\pi_- \in \xbar{B}_{n-m} (a_2+a_4+\dots)$ and (\ref{eq:conseqofhsws}) follows.
\end{proof}

A partition into distinct integers of a positive integer $A$ is a way to write $A$ as a sum of distinct positive integers. Two partitions are considered the same if they differ only in the order of their summands. Denote by $P_D(A)$ the number of different partitions of the integer $A\geq 1$. For consistency let $P_D(0)=1$. {\sc Hardy and Ramanujan} showed in \cite{MR2280877} that for $A \rightarrow \infty$:
\begin{equation}\label{eq:ramanujan}
\log P_D(A) \sim \pi \left( \frac{A}{3} \right)^{1/2}.
\end{equation}

\begin{lemma}\label{lem:estimatehsws}
Let $B> \pi \sqrt{d/3}$. Then there is a constant $K> 0$ such that for all $n \geq 0$
\begin{equation}\label{eq:estimatehsws}
h_n \leq P_D(dn) \sum_{m =0}^{n} b_m \xbar{b}_{n-m} \leq K e^{B\sqrt{n}} \beta_{max}^n.
\end{equation}
\end{lemma}

\begin{proof}
The statement trivially holds for $n=0$. In the case $n>0$ application of Lemma \ref{lem:seqofhsws} in expression (\ref{eq:hswinspans}) and exchanging the finite sums yields
\begin{equation*}
h_n \leq \sum_{m=0}^{n} \sum_{k>0} \sum_{a_1> \dots > a_k>0} b_m(a_1+a_3+\dots) \xbar{b}_{n-m} (a_2+a_4+\dots).
\end{equation*}\par

For given $A,B \geq 0$ we want to count the number of occurrences of the summand $b_m(A) \xbar{b}_{n-m}(B)$ in the sum on the right-hand side. Clearly the total number of sequences $a_1>\dots> a_k>0$ with $a_1+a_3+\dots=A$ and $a_2+a_4+ \dots=B$ is bounded from above by the number $P_D(A+B)$. Using that the height distance per step is at most $d$, it follows that
\begin{equation*}
h_n \leq \sum_{m=0}^{n} \sum_{A=0}^{dm} \sum_{B=0}^{d(n-m)} P_D(A+B) b_m(A) \xbar{b}_{n-m}(B).
\end{equation*}
From $P_D(A+B)\leq P_D(dn)$ and (\ref{eq:bridgessum}) the first inequality in (\ref{eq:estimatehsws}) follows:
\begin{equation*}
\begin{split}
h_n &\leq P_D(dn) \sum_{m=0}^{n} \left( \sum_{A=0}^{dm} b_m(A)\right) \left( \sum_{B=0}^{d(n-m)} \xbar{b}_{n-m}(B)\right)\\ &= P_D (dn) \sum_{m=0}^{n} b_m \xbar{b}_{n-m}.
\end{split}
\end{equation*}
The second inequality in (\ref{eq:estimatehsws}) follows from $b_n \leq \beta^n$ (see (\ref{eq:betamuest})) and the existence of a constant $K>0$ such that
\begin{equation*}
(n+1)P_D(dn) \leq Ke^{B\sqrt{n}}
\end{equation*}
for every $n>0$, which is a consequence of (\ref{eq:ramanujan}).
\end{proof}

\begin{lemma} \label{lem:cnestimate}
Let $C> \pi \sqrt{2d/3}$. Then there is an integer $N$ such that for all $n \geq N$
\begin{equation}\label{eq:cnestimate}
c_n \leq e^{C\sqrt{n+1}} \beta_{max}^{n+1}.
\end{equation}

\end{lemma}
\begin{proof}
Let $\pi=(v_0, \dots, v_n)$ be any SAW of length $n$ and $l$ the maximal index in $[0,n]$ such that $h(v_l)=h_{min}(\pi)$. By the definition of graph height functions there is a neighbour $v'$ of $v_l$ with $h(v')<h(v_l)$. Hence $(v_l,v_{l+1}, \dots, v_n)$ and $(v',v_l,v_{l-1}, \dots, v_0)$ are HSWs in $H_{n-l}$ and $H_{l+1}$ respectively. This construction is shown in Figure \ref{fig:decompsaw} and yields 
\begin{align}
c_n &\leq \sum_{l=0}^n h_{n-l} h_{l+1}.
\end{align}
\begin{figure}[ht]
\centering
\begin{tikzpicture}
\begin{scope}[scale=1.2]
\foreach \i in {0,1,2,3,4}{
\draw[dotted] (0,\i*1) -- (7,\i*1);
}
\draw[-Latex] (0,0) -- (0,4.5);
\fill (1,3) circle (1.5pt);
\fill (2.5,1) circle (1.5pt);
\fill (3,1) circle (1.5pt);
\fill (3.4,1.4) circle (1.5pt);
\fill (5,3) circle (1.5pt);
\fill (3,0.5) circle (1.5pt);

\node at (-0.5,4.5) {$h$};
\node at (1,2.7) {$v_0$};
\node at (2.5,0.7) {$v_{l-1}$};
\node at (3,1.3) {$v_l$};
\node at (3.8,1.2) {$v_{l+1}$};
\node at (5,3.3) {$v_n$};
\node at (3.3,0.5) {$v'$};
\draw[line width=0.8pt] (1,3) to[out=45, in=180] (2,3.5) to[out=0, in=90] (2.5,2.5) to[out=-90, in=90] (1.5,1.5) to[out=-90, in=180] (2.5,1) to (3,1) to (3.4,1.4) to[out=45, in=180] (4,3) to[out=0, in=180] (5.5,1.5) to[out=0, in=-90] (6,2.3) to[out=90, in=-45] (5,3);
\draw[line width=0.8pt] (3,0.5)--(3,1);
\end{scope}
\end{tikzpicture}
	\caption{Decomposition of a SAW into two HSWs.}
	\label{fig:decompsaw}
\end{figure}
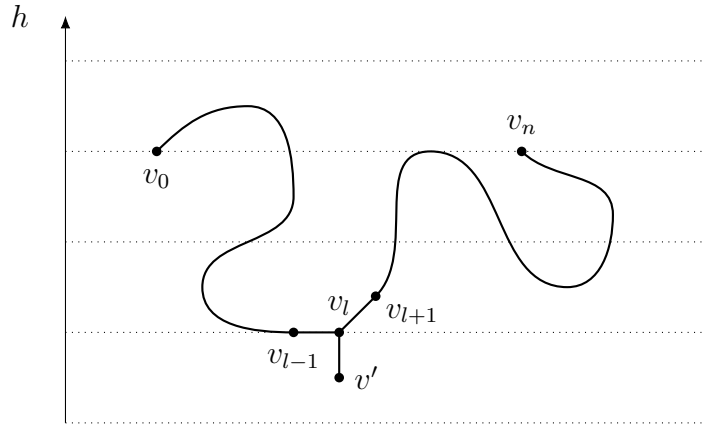

Let $\epsilon>0$ such that $C-\epsilon > \pi \sqrt{2d/3}$. By Lemma \ref{lem:estimatehsws} there is a $K>0$ such that for every $n\geq0$
\begin{equation*}
c_n \leq \sum_{l=0}^n K^2 \exp \left(\frac{C-\epsilon}{\sqrt{2}}\left(\sqrt{n-l}+\sqrt{l+1}\right)\right)  \beta_{max}^{n+1}.
\end{equation*}
Using this estimate and the inequality $\sqrt{\vphantom{b}a}+\sqrt{b} \leq \sqrt{2a+2b}$, which holds for all $a,b \in \mathbb{R}^+$ we obtain
\begin{align}
c_n \leq (n+1) K^2 \exp \left( (C-\epsilon) \sqrt{n+1} \right) \beta_{max}^{n+1}.
\end{align}
For $n$ large enough (\ref{eq:cnestimate}) follows.
\end{proof}

We are now able to finish the proof of Theorem \ref{thm:extbridgethm}. 
Using $\mu^n \leq c_n$ (see (\ref{eq:betamuest})) and Lemma \ref{lem:cnestimate} it follows that for $C > \pi \sqrt{2d/3}$ and $n$ large enough
\begin{equation*}
\mu^{n-1} e^{-C \sqrt{n}} \leq c_{n-1} e^{-C \sqrt{n}} \leq \beta_{max}^n.
\end{equation*}
Applying the $n$-th root and sending $n$ to infinity yields
\begin{equation*}
\mu \leq \beta_{max},
\end{equation*}
finishing the proof of Theorem \ref{thm:extbridgethm} in the transitive case.

We will now briefly discuss the additional steps required to generalize the proof to the case where $\Gamma$ acts quasi-transitively on the graph $X$. For this we need a few additional definitions and results from \cite{MR3862645}.\par

Let the action of $\Gamma$ on $X$ admit $M$ orbits and $\{o_1, \dots, o_M\}$ be a system of representatives of the orbits. Let $r=r(h,\Gamma)$ be the smallest non-negative integer such that for any $0 \leq i,j\leq M$ there is some $v_j \in \Gamma o_j $ and a bridge $\nu(i,j)$ of length at most $r$ starting at $o_i$ and ending at $v_j$, such that $v_j$ is the unique vertex of maximal height in the walk. The walk obtained by going along $\nu(j,i)$ in the reversed direction (from $v_j$ to $o_i$) is a reversed bridge and will be denoted by $\xbar{\nu}(i,j)$.
It has been shown in \cite{MR3862645} (Propositions 3.2 and 4.2)  that $r(h,\Gamma)$ exists for any graph height function and moreover that for any $v \in X$
\begin{equation}\label{eq:bridgesbeta}
b_{n,v}\leq \beta^{n+r}, \quad n\geq 0.
\end{equation}

From now on we denote for $v \in X$ and $a \geq 0$ by $B_{n,v}(a)$ the set of bridges of span $a$ starting at $v$ and by $b_{n,v}(a)$ the cardinality of this set. 

Furthermore, for any $v \in X$ and $a_1> \dots >a_k>0$ let $H_{n,v}(a_1, \dots, a_k)$ be the set of HSWs in $H_{n,v}$, which can be decomposed into an alternating sequence of bridges and reversed bridges of spans $a_1, \dots, a_k$ as described in the transitive case and let $h_{n,v}(a_1, \dots, a_k)$ be its cardinality. Lemma \ref{lem:seqofhsws} can be replaced by the following:

\begin{lemma}\label{lem:qthswdecomp}
Let $n,k \geq 1$, $a_1>a_2 >\dots > a_k >0$ and $v \in X$. Then
\begin{equation}\label{eq:qthswdecomp}
\begin{split}
h_{n,v}(a_1, \dots, a_k) \leq (r+1)^{k-1} \sum_{m=0}^n \left( \sum_{s=0}^{k r} \sum_{t\geq 0} b_{m+s,v}(a_1+a_3+\dots+t)\right) \\ \left( \sum_{s'=0}^{k r} \sum_{t'\geq 0} \xbar{b}_{n-m+s',v}(a_2+a_4+\dots+t')\right).
\end{split}
\end{equation} 
\end{lemma}
\begin{proof}
Given a walk $\pi$ in $H_{n,v}(a_1, \dots, a_k)$, we decompose it into the alternating sequence $(\pi_1, \pi_2, \dots, \pi_k)$ of bridges and reversed bridges as in the transitive case. The main difficulty is that it is not always possible to concatenate the bridges $\pi_l$ and $\pi_{l+2}$ directly, as $\pi_l^+$ and $\pi_{l+2}^-$ may lie in different orbits. Let $(i(0), \dots, i(k-2))$ and $(j(0), \dots, j(k-2))$ be sequences defined such that $v \in \Gamma o_{i(0)}$, $\pi_l^+ \in \Gamma o_{i(l)}$ for $l\geq 1$ and $\pi_{l+2}^- \in \Gamma o_{j(l)}$ for $l \geq 0$. Let 
\begin{equation*}
\nu_l=\begin{cases} \nu(i(l),j(l)) \quad \text{ if $l$ is odd,} \\
\xbar{\nu}(i(l),j(l))\quad \text{ if $l$ is even}. 
\end{cases}
\end{equation*}
Define $\pi_+$ as the concatenation of the bridges $\pi_1, \nu_1, \pi_3, \nu_3, \dots$ (odd indices) and $\pi_-$ as the concatenation of the reversed bridges $\nu_0, \pi_2, \nu_2, \pi_4, \dots$ (even indices). This construction can be seen in Figure \ref{fig:qtconcatbridges}.\par 
\begin{figure}[ht]
\centering
\begin{tikzpicture}
\begin{scope}[scale=0.65]
\foreach \i in {-5,-4,-3,-2,-1,0,1,2,3,4,5,6}{
\draw[dotted] (0,\i*1) -- (11,\i*1);
}
\draw[-Latex] (0,-5.5) -- (0,7);
\draw (0,0) -- (11,0);
\fill (2,0) circle (3pt);
\fill (5,4) circle (3pt);
\fill (8,1) circle (3pt);
\fill (9,3) circle (3pt);
\fill (10,2) circle (3pt);

\node at (-0.5,7) {$h$};
\node at (1.7,0.3) {$v$};
\node at (3.5,3.2) {$\pi_1$};
\node at (6.5,1.5) {$\pi_2$};
\node at (8.9,1.6) {$\pi_3$};
\node at (10,2.5) {$\pi_4$};
\node at (2.2,-0.5) {$\nu_0$};

\draw[line width=1pt] (2,0) to[out=90, in=180] (3,2) to[out=0, in=180] (4,1) to[out=0, in=-90] (4.5,2) to[out=90, in=-90] (1,3.5) to[out=90, in=180] (2,4) to[out=0, in=180] (5,4);
\draw[line width=1pt] (5,4) to[out=0, in=180] (6.5,2) to[out=0, in=180] (7,3) to[out=0, in=180] (8,1); 
\draw[line width=1pt] (8,1) to[out=0, in=180] (9,3); 
\draw[line width=1pt] (9,3) to[out=0, in=180] (10,2);
\draw[line width=1pt,densely dotted] (5,4) -- ++(0.7,0.7);
\draw[line width=1pt,densely dotted] (2,0) -- ++(-0.7,-0.7);
\node at (5,4.5) {$\nu_1$};

\begin{scope}[xshift=0.7cm, yshift=0.7cm]
\draw[line width=1pt, dashed] (5,4) to[out=0, in=180] (6,6);
\node at (5.9,4.8) {$\pi_3$};
\fill (6,6) circle (3pt);
\fill (5,4) circle (3pt);
\end{scope}

\begin{scope}[xshift=-0.7cm,yshift=-0.7cm]
\draw[line width=1pt, dashed] (2,0) to[out=0, in=180] (3.5,-2) to[out=0, in=180] (4,-1) to[out=0, in=180] (5,-3);
\fill (2,0) circle (3pt);
\fill (5,-3) circle (3pt);
\draw[line width=1pt,densely dotted] (5,-3) -- ++(0.7,-0.7);
\node at (4.9,-3.7) {$\nu_2$};
\node at (2.7,-2) {$\pi_2$};
\end{scope}

\begin{scope}[xshift=0cm,yshift=-1.4cm]  
\draw[line width=1pt, dashed] (5,-3) to[out=0, in=180] (6,-4);
\fill (6,-4) circle (3pt);
\fill (5,-3) circle (3pt);
\node at (6,-3.3) {$\pi_4$};
\end{scope}
\end{scope}
\end{tikzpicture}
	\caption{Construction of $\pi_+$ and $\pi_-$, dotted lines are $\nu$-walks.}
	\label{fig:qtconcatbridges}
\end{figure}
Let $m$ be the sum of the lengths of $\pi_i$ having odd index $i$. Then $\pi_+$ is in $B_{m+s,v}(a_1+a_3+\dots+t)$ for some $0 \leq s \leq kr$ and $t\geq 0$ and $\pi_-$ is in $B_{n-m+s',v}(a_2+a_4+\dots+t')$ for some $0 \leq s' \leq kr$ and $t'\geq 0$ as every $\nu$-walk has length at most $r$. The construction is not injective because for a given pair $(\pi_+,\pi_-)$ we cannot directly identify the contained $\nu$-walks. However, the length of any $\nu$-walk is in $[0,r]$, so there are at most $(r+1)$ possibilities per $\nu$-walk. Therefore any pair $(\pi_+,\pi_-)$ can be constructed at most $(r+1)^{k-1}$ times and (\ref{eq:qthswdecomp}) follows.
\end{proof}

Lemma \ref{lem:estimatehsws} is replaced by the following:
\begin{lemma}
There is a constant $B>0$ such that for any $n \geq 0$ and $v \in X$ 
\begin{equation}\label{eq:qthswestimate}
h_{n,v} \leq e^{B \sqrt{n}} \beta_{max}^n.
\end{equation}
\end{lemma}
\begin{proof}
We begin with the observation that for all integers $a,l \geq 0$ 
\begin{equation}\label{eq:qtbridgeswitht}
\sum_{t\geq 0} b_{l,v}(a+t)\leq b_{l,v}.
\end{equation}
Starting with (\ref{eq:hswinspans}) and using Lemma \ref{lem:qthswdecomp} and (\ref{eq:qtbridgeswitht}) yields 
\begin{equation*}
h_{n,v}\leq \sum_{k>0} \sum_{\substack{a_1>\dots>a_k>0\\a_1+\dots+a_k\leq dn}} (r+1)^k \sum_{m=0}^n \left(\sum_{s=0}^{kr} b_{m+s,v}\right) \left(\sum_{s'=0}^{kr} \xbar{b}_{n-m+s',v}\right).
\end{equation*}
Any partition of an integer $A$ into $k$ distinct integers satisfies $k(k+1)\leq 2A$ and therefore $k < \sqrt{2A}$. This fact together with statement (\ref{eq:bridgesbeta}) and the easy observation $\beta_{max} \leq \Delta$ implies
\begin{equation*}
\begin{split}
h_{n,v}&\leq \sum_{k>0} \sum_{\substack{a_1>\dots>a_k>0\\a_1+\dots+a_k\leq dn}} (r+1)^k (n+1) (kr+1)^2 \beta_{max}^{n+2kr+2r}\\
&\leq d n P_D(dn) (r+1)^{\sqrt{2dn}}(n+1) (\sqrt{2dn}r+1)^2 \Delta^{2 \sqrt{2dn}r+2r} \beta_{max}^n.
\end{split}
\end{equation*}
Using (\ref{eq:ramanujan}), for $B$ large enough, (\ref{eq:qthswestimate}) follows.
\end{proof}
Finally we obtain the analogue to Lemma \ref{lem:cnestimate}, which which can be proved in the same way.
\begin{lemma} \label{lem:qtcnestimate}
There is a constant $C>0$ such that for any $n \geq 0$ and $v \in X$
\begin{equation}\label{eq:qtcnestimate}
c_{n,v} \leq e^{C \sqrt{n}} \beta_{max}^{n}.
\end{equation}
\end{lemma}
From this statement it follows directly that $\mu \leq \beta_{max}$, which finishes the proof of Theorem \ref{thm:extbridgethm}.




\section{Bridges in the Grandparent graph}\label{sec:grandparent}

In this section we provide an example of a graph which does not possess a unimodular graph height function. We calculate the bridge constants and use the bridge theorem to obtain the connective constant. \par
An end of a tree is an equivalence class of one-way infinite SAWs, where two walks are equivalent if they share all but finitely many initial vertices. Fix some end $\omega$ of the infinite 3-regular tree $T_3$ and let the graph "hang down" from the end $\omega$. Then the graph can be seen as the union of horizontal layers $H_k$, $k \in \mathbb{Z}$, every vertex $v\in H_k$ is adjacent to one vertex in $H_{k-1}$, called predecessor of $v$ and denoted by $p(v)$, and two vertices in $H_{k+1}$, called successors of $v$. We write $p^{k}(v)$ for the $k$-th predecessor of $v$, i.e. the vertex obtained by $k$ times application of $p$ to $v$. Similarly we denote by $S^{k}(v)$ the set of all vertices $u$ such that $p^{k}(u)=v$. We add the additional pairs $(v,p^{2}(v))$ to the neighbourhood relation of the graph and end up with the graph in Figure \ref{fig:grandparent}, where the newly related vertices are connected by dashed edges. This graph is often called the Grandparent graph and we will denote it by $GP$.\par
\begin{figure}[ht]
	\centering
	\begin{tikzpicture}
	\begin{scope}[scale=1.2]
	\draw[dotted] (-2.5,2) -- (2.5,2);	
	\draw[dotted] (-2.5,0.5) -- (2.5,0.5);
	\draw[dotted] (-2.5,-1) -- (2.5,-1);
	\draw[dotted] (-2.5,-2) -- (2.5,-2);
	\draw[dotted] (-2.5,-2.5) -- (2.5,-2.5);
	\coordinate (0) at (0,0.5);
	\fill (0) circle (1pt);
	\coordinate (00) at (1,-1);
	\fill (00) circle (1pt);
	\coordinate (01) at (-1,-1);
	\fill (01) circle (1pt);
	\coordinate (000) at (1.5,-2);
	\fill (000) circle (1pt);
	\coordinate (001) at (0.5,-2);
	\fill (001) circle (1pt);
	\coordinate (010) at (-0.5,-2);
	\fill (010) circle (1pt);
	\coordinate (011) at (-1.5,-2);
	\fill (011) circle (1pt);
	\coordinate (0000) at (7/4,-2.5);
	\fill (0000) circle (1pt);
	\coordinate (0001) at (5/4,-2.5);
	\fill (0001) circle (1pt);
	\coordinate (0010) at (3/4,-2.5);
	\fill (0010) circle (1pt);
	\coordinate (0011) at (1/4,-2.5);
	\fill (0011) circle (1pt);
	\coordinate (0100) at (-1/4,-2.5);
	\fill (0100) circle (1pt);
	\coordinate (0101) at (-3/4,-2.5);
	\fill (0101) circle (1pt);
	\coordinate (0110) at (-5/4,-2.5);
	\fill (0110) circle (1pt);
	\coordinate (0111) at (-7/4,-2.5);
	\fill (0111) circle (1pt);
	\coordinate (C) at (1.3,2);
	\fill (C) circle (1pt);
	\draw[line width=0.7pt] (C) -- (0);
	\foreach \i in {0,1}{
	\draw[line width=0.7pt] (0) -- (0\i);
	\foreach \j in {0,1}{
	\draw[line width=0.7pt] (0\i) -- (0\i\j);
	\foreach \k in {0,1}{
	\draw[line width=0.7pt] (0\i\j) -- (0\i\j\k);
	}}
	\draw[dashed,line width=0.7pt] (0\i) to (0\i01);
	\draw[dashed,line width=0.7pt] (0\i) to[out=-30,in=70] (0\i00);
	\draw[dashed,line width=0.7pt] (0\i) to (0\i10);
	\draw[dashed,line width=0.7pt] (0\i) to[out=210,in=110] (0\i11);}
	\draw[dashed,line width=0.7pt] (0) to[out=-10,in=80] (000);
	\draw[dashed,line width=0.7pt] (0) to (001);
	\draw[dashed,line width=0.7pt] (0) to (010);
	\draw[dashed,line width=0.7pt] (0) to[out=190,in=100] (011);
	\draw[dashed,line width=0.7pt] (C) to (00);
	\draw[dashed,line width=0.7pt] (C) to[out=190,in=90] (01);
	\draw[->,line width=0.7pt] (C) to node[above] {$\omega$} (2,2.5);
	\draw[line width =0.7] (C) -- ++(-50:1);
	\node[anchor=east] at (-2.6,2) {$H_{k-2}$};
	\node[anchor=east] at (-2.6,0.5) {$H_{k-1}$};
	\node[anchor=east] at (-2.6,-1) {$H_k$};
	\node[anchor=east] at (-2.6,-2) {$H_{k+1}$};
	\node[anchor=east] at (-2.6,-2.5) {$H_{k+2}$};
	\node at (-1.15,-0.85) {$v$};
	\node at (-0.15,0.73) {$p(v)$};
	\node at (-2.1,-1.8) {$S(v)$};
	\draw[gray] (-1.7,-1.8) rectangle (-0.3,-2.2);
	\end{scope}
	\end{tikzpicture}
	\caption{Grandparent graph $GP$.}
	\label{fig:grandparent}
\end{figure}
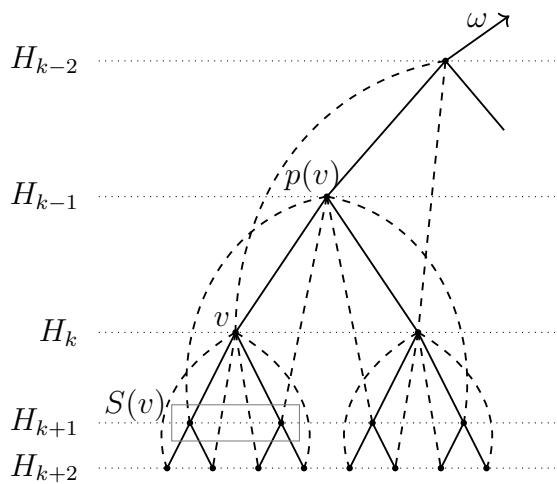
$GP$ is transitive and it is not a Cayley graph of any group, a proof for this was given by {\sc Soardi and Woess} in \cite{MR1082868}. Moreover it is not hard to see that all graph automorphisms on $GP$ fix the end $\omega$. For any vertex $v$ of $X$
\begin{equation*}
\abs{\AUT(GP)_v p(v)}= 1 \neq 2 = \abs{\AUT(GP)_{p(v)} v},
\end{equation*} 
thus the action of $\AUT(GP)$ is not unimodular. This implies that there cannot exist a unimodular graph height function on $GP$.\par
$GP$ admits the obvious graph height function $(h,\AUT(GP))$, where the map $h$ associates to every vertex $v$ the index $k$ of the layer $H_k$ containing $v$.\par
We fix a vertex $o$ and use ordinary generating functions to count bridges starting at $o$ and calculate the bridge constants $\beta(GP,h)$ and $\xbar{\beta}(GP,h)$.
From the structure of the graph it is intuitively clear that there are more bridges then reversed bridges starting at $o$, so we start by counting bridges. 
Let for $a \geq 0$ $\mathcal{B}_{a,v}$ be the set of all bridges of span $a$ starting at a vertex $v$ and $\mathcal{B}_a(x)$ be the ordinary generating functions corresponding to this class, which is independent of $v$ and given by
\begin{equation*}
\mathcal{B}_a(x)= \sum_{n\geq 0} b_n(a) x^n.
\end{equation*}

Then $\mathcal{B}_a(x)$ is a polynomial for every $a \geq 0$ and it is not hard to obtain
\begin{equation*} 
\mathcal{B}_0(x)=1, \quad \mathcal{B}_1(x)=2x, \quad \mathcal{B}_2(x)=4x+4x^2+4x^3.
\end{equation*}
By (\ref{eq:bridgessum}), the generating function $\mathcal{B}(x)$ counting all bridges is
\begin{equation*}
\mathcal{B}(x):= \sum_{n \geq 0} b_n x^n= \sum_{a\geq 0} \mathcal{B}_a(x).
\end{equation*}
By Cauchy-Hadamard's formula the bridge constant $\beta$ is the reciprocal of the radius of convergence of the generating function $\mathcal{B}(x)$.\par
For $a \geq 3$ we recursively count all bridges $\pi=(v_0,v_1,\dots, v_m) \in \mathcal{B}_{a,o}$. The different classes of bridges discussed are shown in Figure \ref{fig:typesofwalks}. The bridge $\pi$ starts at $v_0=o$, so either $v_1 \in S(o)$ or $v_1 \in S^2(o)$. If $v_1 \in S(o)$, then the rest $(v_1,\dots,v_m)$ of $\pi$ must be in $\mathcal{B}_{a-1,v_1}$ (class 1).\par 
Let now $v_1 \in S^2(o)$. We distinguish the following sub-cases:
If $\pi$ does not contain $p(v_1)$, the walk $(v_1,\dots,v_m)$ must be in $\mathcal{B}_{a-2,v_1}$ (class 2).\par
Otherwise, there is some index $l \in [2,n]$ with $v_l=p(v_1)$ and we can decompose the walk at $v_l$ to obtain walks $\pi_1=(v_0, \dots, v_l)$ and $\pi_2=(v_l, \dots v_m)$. This means that $\pi_1$ can have one of two possible shapes, depending on the parity of $l$. Let $l=2k$ in the case where $l$ is even and $l=2k+1$ otherwise. In both cases we have $v_i \in S^2(v_{i-1})$ for $1 \leq  i \leq k-1$. The $k$-th step satisfies $v_k=p(v_{k-1})$ if $l$ is odd and $v_k \in S(v_{k-1})$ for even $l$. The walk concludes with the steps $v_i = p^2(v_{i-1})$ for $k+1 \leq i \leq l$. We call walks with this shape U-walks.\par
Since the span of $\pi$ is $a$, the span of the U-walk $\pi_1$ can be at most $a$. Note that for any $v$ there are 2 vertices in $S(v)$ and 4 vertices in $S^2(v)$. The generating function of U-walks of span at most $a$ is thus given by
\begin{equation*}
\mathcal{U}_a(x)=4x^2+8x^3+\dots+(2x)^a.
\end{equation*}\par
There are three possible ways how the second part $\pi_2$ of $\pi$ may look:\par
\begin{enumerate}
    \item[(i)] $\pi_2 \in \mathcal{B}_{a-1,p(v_1)}$ and does not contain $v_1$ or a vertex in $S(v_1)$. Precisely half of the bridges in $\mathcal{B}_{a-1,p(v_1)}$ satisfy this condition. (class 3)
    \item[(ii)] If $l\geq 3$ then $v_{l-1} \in S(v_1)$, hence $v_{l+1}$ is the other vertex in $S(v_1)$ and $(v_{l+2}, \dots, v_m) \in \mathcal{B}_{a-3,p(v_{l+2})}$. (class 4)
    \item[(iii)] If $l=2$, both of the vertices in $S(v_1)$  are available for $v_{l+1}$ and again $(v_{l+2}, \dots, v_m) \in \mathcal{B}_{a-3,p(v_{l+2})}$. (class 5)
\end{enumerate}
\par
\par

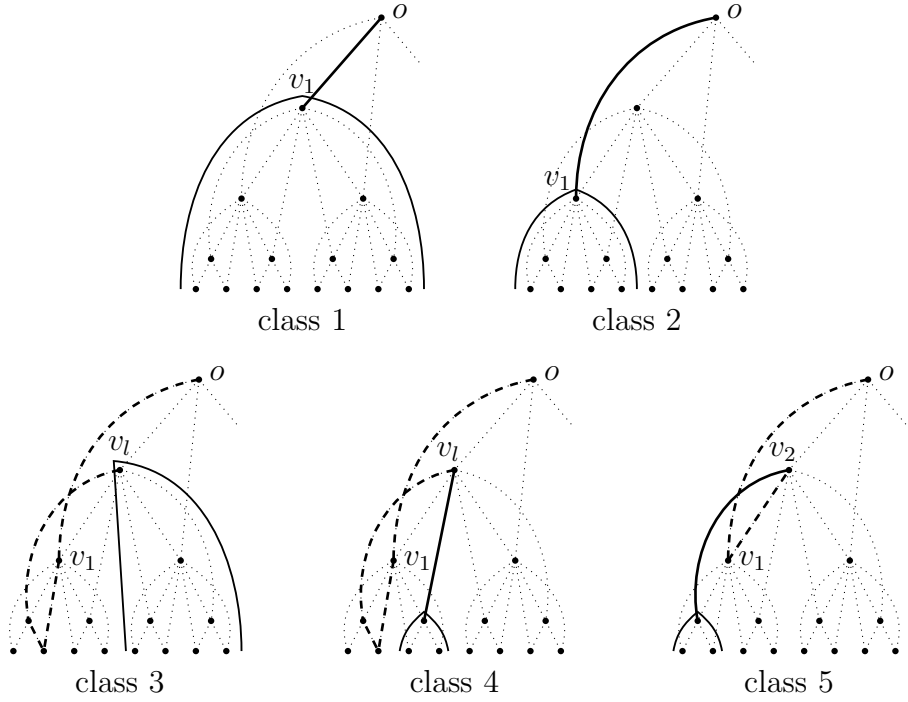
\begin{figure}[ht]
	\centering
	\begin{tikzpicture}
	\begin{scope}[scale=0.8]
	\coordinate (0) at (0,0.5);
	\fill (0) circle (1.5pt);
	\coordinate (00) at (1,-1);
	\fill (00) circle (1.5pt);
	\coordinate (01) at (-1,-1);
	\fill (01) circle (1.5pt);
	\coordinate (000) at (1.5,-2);
	\fill (000) circle (1.5pt);
	\coordinate (001) at (0.5,-2);
	\fill (001) circle (1.5pt);
	\coordinate (010) at (-0.5,-2);
	\fill (010) circle (1.5pt);
	\coordinate (011) at (-1.5,-2);
	\fill (011) circle (1.5pt);
	\coordinate (0000) at (7/4,-2.5);
	\fill (0000) circle (1.5pt);
	\coordinate (0001) at (5/4,-2.5);
	\fill (0001) circle (1.5pt);
	\coordinate (0010) at (3/4,-2.5);
	\fill (0010) circle (1.5pt);
	\coordinate (0011) at (1/4,-2.5);
	\fill (0011) circle (1.5pt);
	\coordinate (0100) at (-1/4,-2.5);
	\fill (0100) circle (1.5pt);
	\coordinate (0101) at (-3/4,-2.5);
	\fill (0101) circle (1.5pt);
	\coordinate (0110) at (-5/4,-2.5);
	\fill (0110) circle (1.5pt);
	\coordinate (0111) at (-7/4,-2.5);
	\fill (0111) circle (1.5pt);
	\coordinate (C) at (1.3,2);
	\fill (C) circle (1.5pt);
	\draw[dotted] (C) -- (0);
	\foreach \i in {0,1}{
	\draw[dotted] (0) -- (0\i);
	\foreach \j in {0,1}{
	\draw[dotted] (0\i) -- (0\i\j);
	\foreach \k in {0,1}{
	\draw[dotted] (0\i\j) -- (0\i\j\k);
	}}
	\draw[dotted] (0\i) to (0\i01);
	\draw[dotted] (0\i) to[out=-30,in=70] (0\i00);
	\draw[dotted] (0\i) to (0\i10);
	\draw[dotted] (0\i) to[out=210,in=110] (0\i11);}
	\draw[dotted] (0) to[out=-10,in=80] (000);
	\draw[dotted] (0) to (001);
	\draw[dotted] (0) to (010);
	\draw[dotted] (0) to[out=190,in=100] (011);
	\draw[dotted] (C) to (00);
	\draw[dotted] (C) to[out=190,in=90] (01);
	\draw[dotted] (C) -- ++(-50:1);
	\draw[line width=1pt] (C) -- (0);
	\draw [line width=0.7pt] (-2,-2.5) to[out=90,in=190] (0,0.7) to[out=-10,in=90] (2,-2.5);
	\node at (1.6,2.1) {$o$};
	\node at (0,0.9) {$v_1$};
	\node at (0,-3) {class 1};
	\end{scope}
	
	\begin{scope}[scale=0.8, xshift=5.5cm]
	\coordinate (0) at (0,0.5);
	\fill (0) circle (1.5pt);
	\coordinate (00) at (1,-1);
	\fill (00) circle (1.5pt);
	\coordinate (01) at (-1,-1);
	\fill (01) circle (1.5pt);
	\coordinate (000) at (1.5,-2);
	\fill (000) circle (1.5pt);
	\coordinate (001) at (0.5,-2);
	\fill (001) circle (1.5pt);
	\coordinate (010) at (-0.5,-2);
	\fill (010) circle (1.5pt);
	\coordinate (011) at (-1.5,-2);
	\fill (011) circle (1.5pt);
	\coordinate (0000) at (7/4,-2.5);
	\fill (0000) circle (1.5pt);
	\coordinate (0001) at (5/4,-2.5);
	\fill (0001) circle (1.5pt);
	\coordinate (0010) at (3/4,-2.5);
	\fill (0010) circle (1.5pt);
	\coordinate (0011) at (1/4,-2.5);
	\fill (0011) circle (1.5pt);
	\coordinate (0100) at (-1/4,-2.5);
	\fill (0100) circle (1.5pt);
	\coordinate (0101) at (-3/4,-2.5);
	\fill (0101) circle (1.5pt);
	\coordinate (0110) at (-5/4,-2.5);
	\fill (0110) circle (1.5pt);
	\coordinate (0111) at (-7/4,-2.5);
	\fill (0111) circle (1.5pt);
	\coordinate (C) at (1.3,2);
	\fill (C) circle (1.5pt);
	\draw[dotted] (C) -- (0);
	\foreach \i in {0,1}{
	\draw[dotted] (0) -- (0\i);
	\foreach \j in {0,1}{
	\draw[dotted] (0\i) -- (0\i\j);
	\foreach \k in {0,1}{
	\draw[dotted] (0\i\j) -- (0\i\j\k);
	}}
	\draw[dotted] (0\i) to (0\i01);
	\draw[dotted] (0\i) to[out=-30,in=70] (0\i00);
	\draw[dotted] (0\i) to (0\i10);
	\draw[dotted] (0\i) to[out=210,in=110] (0\i11);}
	\draw[dotted] (0) to[out=-10,in=80] (000);
	\draw[dotted] (0) to (001);
	\draw[dotted] (0) to (010);
	\draw[dotted] (0) to[out=190,in=100] (011);
	\draw[dotted] (C) to (00);
	\draw[dotted] (C) to[out=190,in=90] (01);
	\draw[dotted] (C) -- ++(-50:1);
	\draw[line width=1pt] (C) to[out=190,in=90] (01);
	\draw [line width=0.7pt] (-2,-2.5) to[out=90,in=200] (-1,-0.85) to[out=-20,in=90] (0,-2.5);
	\node at (1.6,2.1) {$o$};
	\node at (-1.3,-0.7) {$v_1$};
	\node at (0,-3) {class 2};
	\end{scope}
	
	\begin{scope}[scale=0.8, xshift=-3cm,yshift=-6cm]
	\coordinate (0) at (0,0.5);
	\fill (0) circle (1.5pt);
	\coordinate (00) at (1,-1);
	\fill (00) circle (1.5pt);
	\coordinate (01) at (-1,-1);
	\fill (01) circle (1.5pt);
	\coordinate (000) at (1.5,-2);
	\fill (000) circle (1.5pt);
	\coordinate (001) at (0.5,-2);
	\fill (001) circle (1.5pt);
	\coordinate (010) at (-0.5,-2);
	\fill (010) circle (1.5pt);
	\coordinate (011) at (-1.5,-2);
	\fill (011) circle (1.5pt);
	\coordinate (0000) at (7/4,-2.5);
	\fill (0000) circle (1.5pt);
	\coordinate (0001) at (5/4,-2.5);
	\fill (0001) circle (1.5pt);
	\coordinate (0010) at (3/4,-2.5);
	\fill (0010) circle (1.5pt);
	\coordinate (0011) at (1/4,-2.5);
	\fill (0011) circle (1.5pt);
	\coordinate (0100) at (-1/4,-2.5);
	\fill (0100) circle (1.5pt);
	\coordinate (0101) at (-3/4,-2.5);
	\fill (0101) circle (1.5pt);
	\coordinate (0110) at (-5/4,-2.5);
	\fill (0110) circle (1.5pt);
	\coordinate (0111) at (-7/4,-2.5);
	\fill (0111) circle (1.5pt);
	\coordinate (C) at (1.3,2);
	\fill (C) circle (1.5pt);
	\draw[dotted] (C) -- (0);
	\foreach \i in {0,1}{
	\draw[dotted] (0) -- (0\i);
	\foreach \j in {0,1}{
	\draw[dotted] (0\i) -- (0\i\j);
	\foreach \k in {0,1}{
	\draw[dotted] (0\i\j) -- (0\i\j\k);
	}}
	\draw[dotted] (0\i) to (0\i01);
	\draw[dotted] (0\i) to[out=-30,in=70] (0\i00);
	\draw[dotted] (0\i) to (0\i10);
	\draw[dotted] (0\i) to[out=210,in=110] (0\i11);}
	\draw[dotted] (0) to[out=-10,in=80] (000);
	\draw[dotted] (0) to (001);
	\draw[dotted] (0) to (010);
	\draw[dotted] (0) to[out=190,in=100] (011);
	\draw[dotted] (C) to (00);
	\draw[dotted] (C) to[out=190,in=90] (01);
	\draw[dotted] (C) -- ++(-50:1);
	\draw[line width=1pt, dashed] (C) to[out=190,in=90] (01);
	\draw[line width=1pt, dashed]  (01) -- (0110) -- (011);
	\draw[line width=1pt, dashed] (0) to[out=190,in=100] (011);
	\draw [line width=0.7pt] (0.1,-2.5) -- (-0.1,0.65) to[out=-5,in=90] (2,-2.5);
	\node at (1.6,2.1) {$o$};
	\node at (-0.6,-1) {$v_1$};
	\node at (0,0.9) {$v_l$};
	\node at (0,-3) {class 3};
	\end{scope}
	
		\begin{scope}[scale=0.8, xshift=2.5cm,yshift=-6cm]
	\coordinate (0) at (0,0.5);
	\fill (0) circle (1.5pt);
	\coordinate (00) at (1,-1);
	\fill (00) circle (1.5pt);
	\coordinate (01) at (-1,-1);
	\fill (01) circle (1.5pt);
	\coordinate (000) at (1.5,-2);
	\fill (000) circle (1.5pt);
	\coordinate (001) at (0.5,-2);
	\fill (001) circle (1.5pt);
	\coordinate (010) at (-0.5,-2);
	\fill (010) circle (1.5pt);
	\coordinate (011) at (-1.5,-2);
	\fill (011) circle (1.5pt);
	\coordinate (0000) at (7/4,-2.5);
	\fill (0000) circle (1.5pt);
	\coordinate (0001) at (5/4,-2.5);
	\fill (0001) circle (1.5pt);
	\coordinate (0010) at (3/4,-2.5);
	\fill (0010) circle (1.5pt);
	\coordinate (0011) at (1/4,-2.5);
	\fill (0011) circle (1.5pt);
	\coordinate (0100) at (-1/4,-2.5);
	\fill (0100) circle (1.5pt);
	\coordinate (0101) at (-3/4,-2.5);
	\fill (0101) circle (1.5pt);
	\coordinate (0110) at (-5/4,-2.5);
	\fill (0110) circle (1.5pt);
	\coordinate (0111) at (-7/4,-2.5);
	\fill (0111) circle (1.5pt);
	\coordinate (C) at (1.3,2);
	\fill (C) circle (1.5pt);
	\draw[dotted] (C) -- (0);
	\foreach \i in {0,1}{
	\draw[dotted] (0) -- (0\i);
	\foreach \j in {0,1}{
	\draw[dotted] (0\i) -- (0\i\j);
	\foreach \k in {0,1}{
	\draw[dotted] (0\i\j) -- (0\i\j\k);
	}}
	\draw[dotted] (0\i) to (0\i01);
	\draw[dotted] (0\i) to[out=-30,in=70] (0\i00);
	\draw[dotted] (0\i) to (0\i10);
	\draw[dotted] (0\i) to[out=210,in=110] (0\i11);}
	\draw[dotted] (0) to[out=-10,in=80] (000);
	\draw[dotted] (0) to (001);
	\draw[dotted] (0) to (010);
	\draw[dotted] (0) to[out=190,in=100] (011);
	\draw[dotted] (C) to (00);
	\draw[dotted] (C) to[out=190,in=90] (01);
	\draw[dotted] (C) -- ++(-50:1);
	\draw[line width=1pt, dashed] (C) to[out=190,in=90] (01);
	\draw[line width=1pt, dashed]  (01) -- (0110) -- (011);
	\draw[line width=1pt, dashed] (0) to[out=190,in=100] (011);
	\draw[line width=1pt] (0) -- (010);
	\draw [line width=0.7pt] (-0.9,-2.5) to[out=80,in=220] (-0.5,-1.85) to[out=-40,in=100] (-0.1,-2.5);
	\node at (1.6,2.1) {$o$};
	\node at (-0.6,-1) {$v_1$};
	\node at (-0.1,0.8) {$v_l$};
	\node at (0,-3) {class 4};
	\end{scope}
	
		\begin{scope}[scale=0.8, xshift=8cm,yshift=-6cm]
	\coordinate (0) at (0,0.5);
	\fill (0) circle (1.5pt);
	\coordinate (00) at (1,-1);
	\fill (00) circle (1.5pt);
	\coordinate (01) at (-1,-1);
	\fill (01) circle (1.5pt);
	\coordinate (000) at (1.5,-2);
	\fill (000) circle (1.5pt);
	\coordinate (001) at (0.5,-2);
	\fill (001) circle (1.5pt);
	\coordinate (010) at (-0.5,-2);
	\fill (010) circle (1.5pt);
	\coordinate (011) at (-1.5,-2);
	\fill (011) circle (1.5pt);
	\coordinate (0000) at (7/4,-2.5);
	\fill (0000) circle (1.5pt);
	\coordinate (0001) at (5/4,-2.5);
	\fill (0001) circle (1.5pt);
	\coordinate (0010) at (3/4,-2.5);
	\fill (0010) circle (1.5pt);
	\coordinate (0011) at (1/4,-2.5);
	\fill (0011) circle (1.5pt);
	\coordinate (0100) at (-1/4,-2.5);
	\fill (0100) circle (1.5pt);
	\coordinate (0101) at (-3/4,-2.5);
	\fill (0101) circle (1.5pt);
	\coordinate (0110) at (-5/4,-2.5);
	\fill (0110) circle (1.5pt);
	\coordinate (0111) at (-7/4,-2.5);
	\fill (0111) circle (1.5pt);
	\coordinate (C) at (1.3,2);
	\fill (C) circle (1.5pt);
	\draw[dotted] (C) -- (0);
	\foreach \i in {0,1}{
	\draw[dotted] (0) -- (0\i);
	\foreach \j in {0,1}{
	\draw[dotted] (0\i) -- (0\i\j);
	\foreach \k in {0,1}{
	\draw[dotted] (0\i\j) -- (0\i\j\k);
	}}
	\draw[dotted] (0\i) to (0\i01);
	\draw[dotted] (0\i) to[out=-30,in=70] (0\i00);
	\draw[dotted] (0\i) to (0\i10);
	\draw[dotted] (0\i) to[out=210,in=110] (0\i11);}
	\draw[dotted] (0) to[out=-10,in=80] (000);
	\draw[dotted] (0) to (001);
	\draw[dotted] (0) to (010);
	\draw[dotted] (0) to[out=190,in=100] (011);
	\draw[dotted] (C) to (00);
	\draw[dotted] (C) to[out=190,in=90] (01);
	\draw[dotted] (C) -- ++(-50:1);
	\draw[line width=1pt,dashed] (C) to[out=190,in=90] (01);
	\draw[line width=1pt] (0) to[out=190,in=100] (011);
	\draw[line width=1pt,dashed] (01) -- (0);
	\draw [line width=0.7pt] (-1.9,-2.5) to[out=80,in=220] (-1.5,-1.85) to[out=-40,in=100] (-1.1,-2.5);
	\node at (1.6,2.1) {$o$};
	\node at (-0.6,-1) {$v_1$};
	\node at (-0.1,0.8) {$v_2$};
	\node at (0,-3) {class 5};
	\end{scope}
	\end{tikzpicture}
	\caption{Five classes of walks in $B_a^d$. U-walks are drawn dashed.}
	\label{fig:typesofwalks}
\end{figure}

Translating these combinatorial observations into generating functions yields the recursive formula
\begin{equation*}
\mathcal{B}_a(x)=f_a(x)\mathcal{B}_{a-1}(x)+ g(x) \mathcal{B}_{a-2}(x)+ h_a(x) \mathcal{B}_{a-3}(x),
\end{equation*}
where
\begin{equation*}
\begin{split}
f_a(x)&=2x+2x^2\frac{1-(2x)^{a-1}}{1-2x},\\ 
g(x)&= 4x,\\ 
h_a(x)&=8x^3+8x^4\frac{1-(2x)^{a-2}}{1-2x}.
\end{split}
\end{equation*}\par
Fix some $x_0$ in $\mathbb{R}_{\geq 0}$. Then $(\mathcal{B}^d_a(x_0))_{a \geq 0}$ is a sequence in $\mathbb{R}_{\geq 0}$. It is not hard to see that
\begin{equation*}
\mathcal{B}(x_0)= \sum_{a\geq 0} \mathcal{B}_a(x_0) 
\begin{cases}
<\infty \quad \text{if } \lim \limits_{a \rightarrow \infty} (f_a(x_0)+g(x_0)+h_a(x_0)) < 1, \\
=\infty \quad \text{if } \lim \limits_{a \rightarrow \infty} (f_a(x_0)+g(x_0)+h_a(x_0)) > 1.
\end{cases}
\end{equation*}
From this observation it follows that the radius of convergence of $\mathcal{B}(x)$ is the threshold value for $x_0$, which can be found as the smallest positive root of the polynomial
\begin{equation*}
1-8x+10x^2-8x^3+8x^4.
\end{equation*}
So in particular the bridge constant is an algebraic number, which is approximately
\begin{equation*}
\beta(GP,h) \approx 6.64993.
\end{equation*}\par

A similar construction can be used to calculate the reversed bridge constant. For this we define the generating functions $\xbar{\mathcal{B}}_a(x)$ and $\xbar{\mathcal{B}}(x)$ counting reversed bridges as above. The main difference is that we obtain the following recursive formula by looking at the different cases for the final parts of bridges of span $a$:
\begin{equation*}
\xbar{\mathcal{B}}_a(x)=\xbar{f}_a(x)\xbar{\mathcal{B}}_{a-1}(x)+ \xbar{g}(x) \xbar{\mathcal{B}}_{a-2}(x)+ \xbar{h}_a(x) \xbar{\mathcal{B}}_{a-3}(x),
\end{equation*}
where
\begin{equation*}
\begin{split}
\xbar{f}_a(x)&=x+x^2\frac{1-(2x)^{a-2}}{1-2x},\\ 
\xbar{g}(x)&= x,\\ 
\xbar{h}_a(x)&=x^3+x^4\frac{1-(2x)^{a-3}}{1-2x}.
\end{split}
\end{equation*}\par
The radius of convergence of $\bar{\mathcal{B}}(x)$ is the smallest positive root of the polynomial
\begin{equation*}
1-4x+3x^2-x^3+x^4
\end{equation*}
and its reciprocal is the reversed bridge constant
\begin{equation*}
\xbar{\beta}(GP,h) \approx 3.10380.
\end{equation*}\par
As an application of the bridge theorem (Theorem \ref{thm:extbridgethm}) we obtain the connective constant of $GP$,
\begin{equation*}
\mu(GP)=\max\{\beta(GP,h),\xbar{\beta}(GP,h)\} \approx 6.64993.
\end{equation*}

\bibliographystyle{plain}
\bibliography{latex}

\begin{thebibliography}{10}

\bibitem{MR3025395}
Roland Bauerschmidt, Hugo Duminil-Copin, Jesse Goodman, and Gordon Slade.
\newblock Lectures on self-avoiding walks.
\newblock In {\em Probability and statistical physics in two and more
  dimensions}, volume~15 of {\em Clay Math. Proc.}, pages 395--467. Amer. Math.
  Soc., Providence, RI, 2012.

\bibitem{MR2912714}
Hugo Duminil-Copin and Stanislav Smirnov.
\newblock {The connective constant of the honeycomb lattice equals
  {$\sqrt{2+\sqrt{2}}$}}.
\newblock {\em Ann. of Math. (2)}, 175(3):1653--1665, 2012.

\bibitem{CR1000000}
Paul~J. Flory.
\newblock {\em {Principles of Polymer Chemistry}}.
\newblock Cornell University Press, 1953.

\bibitem{MR3743101}
Geoffrey~R. Grimmett and Zhongyang Li.
\newblock {Self-avoiding walks and amenability}.
\newblock {\em Electron. J. Combin.}, 24(4):Paper 4.38, 24, 2017.

\bibitem{MR3862645}
Geoffrey~R. Grimmett and Zhongyang Li.
\newblock {Locality of connective constants}.
\newblock {\em Discrete Math.}, 341(12):3483--3497, 2018.

\bibitem{MR0091568}
J.~M. Hammersley.
\newblock {Percolation processes. {II}. {T}he connective constant}.
\newblock {\em Proc. Cambridge Philos. Soc.}, 53:642--645, 1957.

\bibitem{MR0139535}
J.~M. Hammersley and D.~J.~A. Welsh.
\newblock {Further results on the rate of convergence to the connective
  constant of the hypercubical lattice}.
\newblock {\em Quart. J. Math. Oxford Ser. (2)}, 13:108--110, 1962.

\bibitem{MR2280877}
G.~H. Hardy and S.~Ramanujan.
\newblock {Asymptotic formul{\ae} for the distribution of integers of various
  types [{P}roc. {L}ondon {M}ath. {S}oc. (2) {\bf 16} (1917), 112--132]}.
\newblock In {\em {Collected papers of {S}rinivasa {R}amanujan}}, pages
  245--261. AMS Chelsea Publ., Providence, RI, 2000.

\bibitem{AR005}
Christian Lindorfer and Wolfgang Woess.
\newblock {The language of self-avoiding walks}.
\newblock 2019.
\newblock preprint available at https://arxiv.org/abs/1903.02368.

\bibitem{MR2986656}
Neal Madras and Gordon Slade.
\newblock {\em {The self-avoiding walk}}.
\newblock {Modern Birkh{\"a}user Classics}. Birkh{\"a}user/Springer, New York,
  2013.
\newblock Reprint of the 1993 original.

\bibitem{MR1082868}
Paolo~M. Soardi and Wolfgang Woess.
\newblock {Amenability, unimodularity, and the spectral radius of random walks
  on infinite graphs}.
\newblock {\em Math. Z.}, 205(3):471--486, 1990.

\bibitem{MR811571}
V.~I. Trofimov.
\newblock {Groups of automorphisms of graphs as topological groups}.
\newblock {\em Mat. Zametki}, 38(3):378--385, 476, 1985.

\end{thebibliography}

\end{document}